\documentclass[11pt,a4paper]{amsart}
\usepackage{amsmath, amsthm, amsfonts, amssymb, amscd}
\usepackage[bookmarks=false]{hyperref}

\newtheorem{theo}{Theorem}[section]
\newtheorem{prop}[theo]{Proposition}
\newtheorem{lem}[theo]{Lemma}

\newtheorem{letterthm}{Theorem}

\newtheorem{lettercor}[letterthm]{Corollary}

\theoremstyle{definition} 
\newtheorem{defs}[theo]{Definition}
\newtheorem*{defs*}{Definition}
\newtheorem{ex}[theo]{Example}	
\newtheorem{rem}[theo]{Remark}

\newcommand{\R}{\mathbb{R}}
\newcommand{\C}{\mathbb{C}}
\newcommand{\Z}{\mathbb{Z}}

\newcommand{\ovt}{\mathbin{\overline{\otimes}}}

\begin{document}

  \title{Solidity of type III Bernoulli crossed products}

\author{Amine Marrakchi}
\address{\'Ecole Normale Sup\'erieure \\ 45 rue d'Ulm 75230 Paris Cedex 05 \\ France}
\address{Laboratoire de Math\'ematiques d'Orsay \\ Universit\'e Paris-Sud \\ Universit\'e Paris-Saclay \\ 91405 Orsay \\ France}
\email{amine.marrakchi@ens.fr}
\thanks{A. Marrakchi is supported by ERC Starting Grant GAN 637601}
	
	\begin{abstract} We generalize a theorem of Chifan and Ioana by proving that for any, possibly type III, amenable von Neumann algebra $A_0$, any faithful normal state $\varphi_0$ and any discrete group $\Gamma$, the associated Bernoulli crossed product von Neumann algebra $M=(A_0,\varphi_0)^{\ovt \Gamma}\rtimes \Gamma$ is solid relatively to $\mathcal{L}(\Gamma)$. In particular, if $\mathcal{L}(\Gamma)$ is solid then $M$ is solid and if $\Gamma$ is non-amenable and $A_0 \neq \mathbb{C}$ then $M$ is a full prime factor. This gives many new examples of solid or prime type $\mathrm{III}$ factors. Following Chifan and Ioana, we also obtain the first examples of solid non-amenable type $\mathrm{III}$ equivalence relations.
		\end{abstract}
		
	
  \maketitle
	
	\bibliographystyle{alpha}	

\section{Introduction}
In \cite{TakaSolid}, Ozawa discovered a remarkable rigidity property of von Neumann algebras that he called \emph{solidity}. One of the many possible definitions is the following one: a von Neumann algebra $M$ is \emph{solid} if for every subalgebra with expectation $Q \subset M$ there exists a sequence of projections $z_n \in \mathcal{Z}(Q)$ with $\sum_n z_n =1$ such that $Qz_0$ is amenable and $Qz_n$ is a non-amenable factor for all $n \geq 1$. The main interest of this notion is that any solid non-amenable factor is automatically \emph{full} (every centralizing sequence is trivial) and \emph{prime} (not a tensor product of two non type $\mathrm{I}$ factors). Ozawa's celebrated result \cite{TakaSolid} states that the group von Neumann algebra $\mathcal{L}(\Gamma)$ of any hyperbolic group $\Gamma$ is solid.

A closely related property in the context of \emph{equivalence relations} was discovered by Chifan and Ioana in \cite{ChifanIoanaBern}. An equivalence relation $\mathcal{R}$ on a probability space $(X,\mu)$ is \emph{solid} (or \emph{solidly ergodic} \cite[Definition 5.4]{GaboriauSurvey}) if for every subequivalence relation $\mathcal{S} \subset \mathcal{R}$ there exists a measurable partition $X=\bigsqcup_n X_n$ by $\mathcal{S}$-invariant subsets such that $\mathcal{S}_{|X_0}$ is amenable and $\mathcal{S}_{|X_n}$ is non-amenable and ergodic for all $n \geq 1$. A solid ergodic non-amenable equivalence relation is automatically \emph{strongly ergodic} (every sequence of almost invariant subsets is trivial) and \emph{prime} (not a product of two non type $\mathrm{I}$ equivalence relations). The main theorem of Chifan and Ioana \cite[Theorem 1]{ChifanIoanaBern} states that the orbital equivalence relation of a Bernoulli action of \emph{any} countable group $\Gamma$ is solid.  In fact, they deduce this solidity result from a stronger theorem \cite[Theorem 2]{ChifanIoanaBern} which essentially says that for any \emph{tracial} amenable von Neumann algebra $(A_0,\tau_0)$ and any discrete group $\Gamma$, the \emph{Bernoulli crossed product von Neumann algebra} $M=(A_0,\tau_0)^{\ovt \Gamma}\rtimes \Gamma$ is \emph{solid relatively to} $\mathcal{L}(\Gamma)$, a notion which we define precisely in section 3. Unlike the method of Ozawa which is based on $C^*$-algebraic techniques and requires the group $\Gamma$ to be exact \cite[Proposition 4.5 and 4.6]{TakaAmenAct}, the approach of Chifan and Ioana is based on Popa's \emph{deformation/rigidity theory} \cite{PopaSurvey} and his \emph{spectral gap rigidity} principle \cite{PopaSpectralGap} so that they do not need to make any assumption on $\Gamma$. Using the same approach, R. Boutonnet was able to generalize their results to \emph{Gaussian actions} \cite{RemiGaussian}. 

Our main theorem generalizes the relative solidity result of Chifan and Ioana to non-tracial, possibly type $\mathrm{III}$, Bernoulli crossed products.

\begin{letterthm} \label{solid}
Let $A_0$ be an amenable von Neumann algebra with separable predual, $\varphi_0$ a faithful normal state on $A_0$ and $\Gamma$ any countable group. Let $M=(A_0,\varphi_0)^{\ovt \Gamma}\rtimes \Gamma$ be the associated Bernoulli crossed product von Neumann algebra. For any subalgebra with expectation $Q \subset 1_QM1_Q$ such that $Q \nprec_M \mathcal{L}(\Gamma)$ (see Section 2 for Popa's intertwining symbol $\prec_M$) there exists a sequence of projections $z_n \in \mathcal{Z}(Q)$ with $\sum_n z_n =1_Q$ such that 
\begin{itemize}
	\item $Qz_0$ is amenable.
	\item $Qz_n$ is a full prime factor for all $n \geq 1$.
\end{itemize}
In particular, if $\Gamma$ is non-amenable and $A_0 \neq \mathbb{C}$ then $M$ is a full prime factor and if $\mathcal{L}(\Gamma)$ is solid then $M$ is solid.
\end{letterthm}
The fullness of type $\mathrm{III}$ Bernoulli crossed products by non-amenable groups was already established by S.Vaes and P.Verraedt in \cite{Vaes2015296}. However, Theorem \ref{solid} provides many new examples of prime or solid type $\mathrm{III}$ factors. In fact, in \cite{ConnesAlmostPeriodic}, A. Connes introduced two new invariants for type $\mathrm{III}$ factors: the $\mathrm{Sd}$ invariant and the $\tau$ invariant. He used the noncommutative Bernoulli crossed products to construct type $\mathrm{III}$ factors with prescribed invariants $\mathrm{Sd}$ and $\tau$. By combining these constructions with Theorem \ref{solid}, we obtain the following corollary:

\begin{lettercor} \label{invariants}
For every countable subgroup $\Lambda \subset \mathbb{R}^*_+$, there exists a solid non-amenable type $\mathrm{III}$ factor with separable predual and with a Cartan subalgebra such that its $\mathrm{Sd}$ invariant is $\Lambda$. For any topology $\tau_0$ on $\mathbb{R}$ induced by an injective continuous separable unitary representation of $\mathbb{R}$, there exists a solid non-amenable type $\mathrm{III}_1$ factor with separable predual and with a Cartan subalgebra such that its $\tau$ invariant is $\tau_0$.
\end{lettercor}

Note that there was previously no known example of a non-amenable solid type $\mathrm{III}$ factor with a Cartan subalgebra. Using Theorem \ref{solid} we can generalize \cite[Theorem 7]{ChifanIoanaBern} by removing the assumption that the base equivalence relation is measure preserving and hence we also obtain the first examples of non-amenable solid type $\mathrm{III}$ equivalence relations.

\begin{lettercor} \label{equivalence}
Let $(X_0,\mu_0)$ be a probability space, $\mathcal{R}_0$ an arbitrary amenable non-singular equivalence relation on $X_0$ and $\Gamma$ any countable group. Let $\mathcal{R}=\mathcal{R}_0 \wr \Gamma$ be the \emph{wreath product equivalence relation} on $(X_0,\mu_0)^{\Gamma}$ defined by $(x_i) \sim_{\mathcal{R}} (y_i)$ if and only if there exists $g \in \Gamma$ and a finite subset $F \subset \Gamma$ such that

\begin{itemize}
	\item $\forall i \in F, \; x_i \sim_{\mathcal{R}_0} y_{gi}$
	\item $\forall i \in \Gamma \setminus F, \; x_i=y_{gi}$
\end{itemize}
Then $\mathcal{R}$ is solid, i.e.\ for any subequivalence relation $\mathcal{S} \subset \mathcal{R}$ there exists a countable partition of $(X_0,\mu_0)^{\Gamma}$ into $\mathcal{S}$-invariant components $Z_n, n \in \mathbb{N}$ such that
\begin{itemize}
	\item $\mathcal{S}_{| Z_0}$ is amenable.
	\item $\mathcal{S}_{| Z_n}$ is strongly ergodic and prime for all $n \geq 1$.
\end{itemize}
\end{lettercor}

This article is organized in the following way. Section 2 is devoted to preliminaries. We recall in particular the theory of ultraproducts for non-tracial von Neumann algebras. We also recall \emph{Popa's intertwining theorem} \cite{PopaBetti} and its recent generalization to arbitrary von Neumann algebras by C. Houdayer and Y. Isono \cite{HoudIsoUPF}. In Section 3, we introduce the notion of \emph{relative solidity} and its main properties. In Section 4, we recall the notion of \emph{s-malleable deformations} \cite{PopaMalleable1} and prove an abstract non-tracial version of Popa's \emph{spectral gap rigidity} argument \cite{PopaSpectralGap} by using ultraproduct techniques. Finally in Section 5, we prove our main theorems. Our proof follows the lines of Chifan and Ioana's original proof. Indeed, we introduce the same s-malleable deformation of the Bernoulli crossed product and we show that the same bimodule computation, and hence the key spectral gap property, still holds in the non-tracial situation. Therefore we can apply the spectral gap argument of Section 4. Houdayer and Isono's generalization of Popa's intertwining theorem is then used to prove a crucial dichotomy for rigid subalgebras, similar to the one used by Chifan and Ioana, from which the conclusion follows easily. Altogether, the proof does not use Takesaki's continuous decomposition of type $\mathrm{III}$ factors.

\subsubsection*{Acknowledgment}
We are very grateful to our advisor C. Houdayer for suggesting this problem and for his help during this work. We also thank R. Boutonnet for his useful comments.

\tableofcontents

\section{Preliminaries}
All mentioned von Neumann algebras $M$ are supposed to be \emph{countably decomposable} (or \emph{$\sigma$-finite}). This means that $M$ admits a faithful normal state. However, we do not assume that the predual $M_*$ is separable unless it is explicitly stated. We say that $M$ is \emph{diffuse} if it has no minimal projection, i.e.\ $pMp \neq \mathbb{C}p$ for every non-zero projection $p \in M$. We say that $M$ is \emph{properly non-amenable} if $Mz$ is non-amenable for every non-zero projection $z \in \mathcal{Z}(M)$. Equivalently $M$ is properly non-amenable if and only if $pMp$ is non-amenable for every non-zero projection $p \in M$.

An inclusion $N \subset M$ of two von Neumann algebras is always assumed to be unital. For non-unital inclusions we will use the notation $N \subset 1_NM1_N$ where the projection $1_N \in M$ is the unit of $N$. An inclusion $N \subset M$ is said to be \emph{with expectation} if there is a faithful normal conditional expectation $E_N: M \rightarrow N$. 

For every von Neumann algebra $M$, the Hilbert space $L^{2}(M)$ denotes the \emph{standard form} of $M$ \cite[Chapter $\mathrm{IX}$, Section $1$]{TakesakiII}. For every $\xi \in L^2(M)$ we use the $M$-$M$-bimodule notation $a\xi b$ for the left and right action of $a,b \in M$. For every normal faithful semi-finite weight $\psi$ on $M$, there is a canonical map, denoted $x \mapsto x \psi^{1/2}$, from $\mathfrak{n}_\psi = \{ x \in M \mid \psi(x^*x) < +\infty \}$ to $L^2(M)$ such that $\langle x \psi^{1/2} , y \psi^{1/2} \rangle = \psi(y^*x)$ for all $x,y \in \mathfrak{n}_\psi$. This map identifies $L^2(M)$ with the completion of $\mathfrak{n}_\psi$ with respect to the inner product $(x,y) \mapsto \psi(y^*x)$. Moreover $\mathfrak{n}_\psi$ is a left ideal in $M$ and the map $x \mapsto x \psi^{1/2}$ is compatible with the left multiplication: $a(x\psi^{1/2})=(ax)\psi^{1/2}$ for every $a \in M, x \in \mathfrak{n}_\psi$. Similarly there is a map $x \mapsto \psi^{1/2}x$ from $\mathfrak{n}_\psi^*=\{ x \in M \mid \psi(xx^*) < + \infty \}$ to $L^2(M)$ with the similar properties. If $\psi$ is finite then $1 \in \mathfrak{n}_\psi=\mathfrak{n}_\psi^*=M$ and $\psi^{1/2}$ makes sense as a vector in $L^2(M)$ in this case.
 
For a von Neumann algebra $M$ with a faithful normal semi-finite weight $\psi$, let $t \mapsto \sigma_t^\psi$ denote the modular flow of $\psi$. An element $x \in M$ is said to be $\psi$-\emph{analytic} if the function $t \mapsto \sigma_t^\psi(x)$ extends to an analytic function defined on the entire complex plane $\mathbb{C}$, which will be still denoted $z \mapsto \sigma_z^\psi(x)$. The $\psi$-analytic elements form a dense $*$-algebra in $M$ (see \cite[Chapter VIII, Lemma 2.3]{TakesakiII}). We let $M^\psi$ denote the centralizer of $\psi$, i.e.\ the fixed point algebra of $\sigma^\psi$. For every $\psi$-analytic element $x \in M$ we have $x\psi^{1/2}=\psi^{1/2} \sigma_{i/2}^\psi(x)$.
 
 If $E_N: M \rightarrow N$ is a faithful normal conditional expectation and $\varphi$ is a faithful normal state on $N$ then for the faithful normal state on $M$ defined by $\widetilde{\varphi}=\varphi \circ E_N$ we have that
\[ \forall x \in N, \; \sigma_t^\varphi(x)=\sigma_t^{\widetilde{\varphi}}(x). \]
In particular $N$ is globally invariant by $\sigma_t^\psi$. Conversely, if $\varphi$ is a faithful normal state on $M$ such that $N$ is globally invariant by $\sigma_t^\varphi$ for all $t \in \mathbb{R}$, then by a theorem of Takesaki \cite[Chapter IX, Theorem 4.2]{TakesakiII} there exists a unique faithful normal conditional expectation from $M$ to $N$ which preserves $\varphi$. In particular, any subalgebra of $M^\varphi$ is with expectation in $M$. Also, we see that if $A \subset M$ is with expectation then $A' \cap M$ and $A \vee (A' \cap M)$ are also subalgebras with expectation in $M$.

Another important consequence is the following lemma which will be used frequently, without explicit reference. It already appears in \cite[Proposition 2.2]{HoudUedaFreeRig} but for the reader's convenience, we reproduce it here with a simpler proof.

\begin{lem}\label{expectation}
Let $A \subset M$ be an inclusion with expectation and let $e$ be a projection in $A$ or in $A' \cap M$. Then $eAe \subset eMe$ is also with expectation.
\end{lem}
\begin{proof}
Let $e \in A$ be a projection and $E: M \rightarrow A$ a faithful normal conditional expectation. Then $E(eMe) \subset eAe$, hence $E$ restricts to a faithful normal conditional expectation from $eMe$ to $eAe$.

Now let $e \in A' \cap M$. First we replace $E$ by a new faithful normal conditional expectation $\tilde{E}: M \rightarrow A$ defined by
\[ \forall x \in M,  \; \tilde{E}(x)=E(exe+(1-e)x(1-e)). \]
Now, pick $\varphi$ a faithful normal state on $A$ and let $\psi = \varphi \circ \tilde{E}$. Then we have $\sigma^\varphi=\sigma^\psi_{|A}$. Also, by construction $e$ is in the centralizer of $\psi$. Now define a new faithful normal state $\phi = \psi_{|eMe}$. Since $e \in M^\psi$, we have $\forall x \in eMe, \; \sigma_t^\phi(x)=\sigma_t^\psi(x)$. In particular, it is then clear that $eAe$ is globally invariant by $\sigma^\phi$ in $eMe$. Hence it is with expectation by Takesaki's theorem.
\end{proof}
  
\subsection*{Ultraproducts}
In this section we recall the construction of ultraproducts for non-tracial von Neumann algebras. A general reference on this topic is \cite{AndoHaagUltra}.

Let $M$ be any von Neumann algebra. Let $\omega$ be any free ultrafilter $\omega \in \beta \mathbb{N} \setminus \mathbb{N}$. In the von Neumann algebra $\ell^\infty(\mathbb{N},M)$ we define a $C^*$-subalgebra
\[ \mathcal{I}^\omega(M)=\{ (x_n)_{n} \in \ell^\infty(\mathbb{N},M) \mid \lim_{n \rightarrow \omega} x_n = 0 \; *\text{-strongly} \}. \] 
When $M$ is finite, $\mathcal{I}^\omega(M)$ is an ideal of $\ell^\infty(\mathbb{N},M)$ and one defines the \emph{ultraproduct} algebra $M^\omega$ as the quotient $\ell^\infty(\mathbb{N},M)/\mathcal{I}^\omega(M)$. 

When $M$ is not finite, $\mathcal{I}^\omega(M)$ is no longer an ideal and one introduces instead its \emph{multiplier algebra}
\[ \mathcal{M}^\omega(M)=\{ (x_n)_{n} \in \ell^\infty(\mathbb{N},M) \mid (x_n)_{n}\mathcal{I}^\omega(M) \subset \mathcal{I}^\omega(M) \text{ and } \mathcal{I}^\omega(M)(x_n)_{n} \subset \mathcal{I}^\omega(M)  \}. \] 
The quotient $M^\omega=\mathcal{M}^\omega(M)/\mathcal{I}^\omega(M)$ is always a von Neumann algebra \cite[Theorem 5.1]{OcneanuUltra} called the \emph{Ocneanu ultraproduct} of $M$. If $(x_n)_n \in \mathcal{M}^\omega(M)$ we denote by $(x_n)^\omega_n$ its image in the quotient $M^\omega$. One can identify $M$ with the algebra of constant sequences in $\mathcal{M}^\omega(M)$ and hence and we have a natural inclusion $M \subset M^\omega$. This inclusion is with expectation since we have a canonical faithful normal conditional expectation $E^\omega: M^\omega \rightarrow M$ defined by
\[ E^\omega((x_n)^\omega_n)=\lim_{n \rightarrow \omega} x_n \; \text{ in the weak* topology.} \]  
If $\varphi$ is a faithful normal state on $M$, we denote by $\varphi^\omega$ the faithful normal state on $M^\omega$ defined by $\varphi^\omega= \varphi \circ E^\omega$. It holds for the modular flow of $\varphi^\omega$ that
\[ \sigma_t^{\varphi^\omega}((x_n)_n^\omega)=(\sigma_t^{\varphi}(x_n))_n^\omega. \]

Let $N \subset M$ be a subalgebra. Then $\ell^\infty(\mathbb{N},N) \subset \ell^\infty(\mathbb{N},M)$ and $\mathcal{I}^\omega(N) \subset \mathcal{I}^\omega(M)$. If $N \subset M$ is with expectation then we also have $\mathcal{M}^\omega(N) \subset \mathcal{M}^\omega(M)$. Hence, in this case, we can identify $N^\omega$ canonically with a von Neumann subalgebra of $M^\omega$. Moreover, the inclusion $N^\omega \subset M^\omega$ is with expectation. Indeed if $E_N: M \rightarrow N$ is a faithful normal conditional expectation, then we can define a faithful normal conditional expectation $E_{N^\omega}: M^\omega \rightarrow N^\omega$ by
\[ E_{N^\omega}((x_n)_n^\omega)=(E_N(x_n))_n^\omega. \] 

\subsection*{Popa's intertwining theory}

In this section we recall the powerful method for intertwining subalgebras developed by S. Popa \cite[Appendix]{PopaBetti} and \cite[Theorem 2.1, Corollary 2.3]{PopaMalleable1}. We will also need the recent generalization of this method to arbitrary von Neumann algebras by C. Houdayer and Y. Isono \cite{HoudIsoUPF}.

The following lemma will be used a lot, without explicit reference. A proof can be found in \cite[Lemma 5.5, Chapter $\mathrm{XIV}$]{TakesakiIII}.
\begin{lem}\label{corner}
Let $A$ be a subalgebra of a von Neumann algebra $M$ and $e$ a projection in $A$ or in $A' \cap M$. We have $e(A' \cap M)e=(eAe)' \cap eMe$ and $e(A \vee (A' \cap M))e=eAe \vee e(A' \cap M)e$.
\end{lem}

\begin{defs} \label{def_semi_conj}
Let $M$ be a von Neumann algebra and let $A \subset 1_AM1_A$ and $B \subset 1_BM1_B$ be two subalgebras with expectations. For $h \in M$, we say that $A$ and $B$ are \emph{semi-conjugated} by $h \in M$ if 
\[ \forall x \in A, \; hx=0 \Rightarrow x=0 \]
 and there exists an (onto) $*$-isomorphism $\psi: A \rightarrow B$ such that 
\[ \forall x \in A, \; \psi(x)h=hx. \]
We denote this relation by $A \sim_h B$. We will also write $A \sim_M B$ if there exists an $h \in M$ such that $A \sim_h B$.
\end{defs}

\begin{rem}\label{semi-conj}
We make the following easy remarks: 
\begin{itemize}
\item The $*$-isomorphism $\psi$ is completely determined by $h$.
\item $A \sim_h B$ if and only if $B \sim_{h^*} A$.  
\item $A \sim_h B$ implies that $A \sim_{1_B h1_A} B$ so we can always suppose that $h \in 1_BM1_A$.
\item If $h=v|h|$ is the polar decomposition of $h$ then $A \sim_h B$ implies that $A \sim_{v} B$ so we can always suppose that $h$ is a partial isometry in $1_BM1_A$. 
\end{itemize}
However, we note a very important difficulty with the relation $\sim_M$: it is \emph{not transitive}. Indeed if $A \sim_h B$ and $B \sim_g C$, then nothing guarantees that $gh \neq 0$. For this, one has to control the relative commutants of the involved subalgebras. For example, an interesting special case is when $A' \cap 1_AM1_A \subset A$ and $B' \cap 1_BM1_B \subset B$ (e.g.\ $A$ and $B$ maximal abelian). Indeed, in this case, if a partial isometry $v \in 1_BM1_A$ satisfies $A \sim_v B$ then we get a genuine conjugacy, i.e.\ $v^*v=1_A$, $vv^*=1_B$ and $B=vAv^*$. In the general case however, we only know that $e=v^*v \in A' \cap 1_AM1_A$ and $f=vv^* \in B' \cap 1_BM1_B$. Using Lemma \ref{corner}, we see that the isomorphism $\mathrm{Ad}(v): eMe \rightarrow fMf$ sends $Ae$ onto $Bf$ and $(Ae)' \cap eMe=e(A'\cap 1_AM1_A)e$ onto $(Bf)' \cap fBf = f(B' \cap 1_BM1_B)f$. In particular, $e(A \vee (A' \cap 1_AM1_A))e$ is conjugated to $f(B \vee (B' \cap 1_BM1_B))f$ by $v$.

\end{rem}

\begin{defs}
Let $M$ be a von Neumann algebra and let $A \subset 1_AM1_A$ and $B \subset 1_BM1_B$ be two subalgebras with expectation. We say that \emph{a corner of $A$ embeds with expectation into a corner of $B$ inside $M$} if there exists non-zero projections $p \in A$ and $q \in B$ and a subalgebra with expectation $C \subset qBq$ such that $pAp \sim_M C$.
\end{defs}

 We will use the notation $A \prec_M B$ to say that a corner of $A$ embeds with expectation into a corner of $B$ inside $M$ and $A \nprec_M B$ to say that no corner of $A$ embeds with expectation into a corner of $B$ inside $M$.	

\begin{rem}
We note that in order to check that $A \prec_M B$ it suffices to find projections $p \in A$, $q \in B$, a normal $*$-morphism $\psi: pAp \rightarrow qBq$ and a non-zero $h \in qMp$ such that $\psi(pAp) \subset qBq$ is with expectation and  $\forall x \in pAp, \psi(x)h=hx$. Indeed, in this case, there is a projection $p' \in \mathcal{Z}(pAp)$ and an injective $*$-morphism $\psi': p'Ap' \rightarrow qBq$ such that $\psi(x)=\psi'(p'x)$ for all $x \in pAp$. And since we have $h(1-p')=\psi(1-p')h=0$ we will have $h \in qMp'$ and $\forall x \in p'Ap', \psi'(x)h=hx$. Therefore, we may always suppose that $\psi$ is injective. Now, we can see $\psi$ as an $*$-isomorphism from $pAp$ to $C$ where $C \subset qBq$ is a subalgebra with expectation. However, we don't have yet the condition $\forall x \in pAp, hx=0 \Rightarrow x=0$. To fix this, note that $\{ x \in pAp \mid hx=0 \}$ is a two-sided ideal in $pAp$ so there exists a unique projeciton $p' \in \mathcal{Z}(pAp)$ such that $hx=0 \Leftrightarrow p'x=0$. We have $p' \neq 0$ because $h \neq 0$. Let $q'=\psi(p') \in \mathcal{Z}(C)$. Let $C'=Cq' \subset q'Bq'$ which is a subalgebra with expectation. Finally let $\psi': p'Ap' \rightarrow C'$ be the restriction of $\psi$. Now we still have $\forall x \in p'Ap', \psi'(x)h=hx$ and the condition $hx=0 \Rightarrow x=0$ is satisfied for all $x \in p'Ap'$ which means that $p'Ap' \sim_h C'$. Hence we have indeed $A \prec_M B$.
\end{rem}

The following nontrivial proposition will be used frequently in the sequel.

\begin{prop}[{\cite[Lemma 4.8]{HoudIsoUPF}}] \label{sub_embed}
Let $M$ be a von Neumann algebra and $A \subset 1_AM1_A$, $B \subset 1_BM1_B$ two subalgebras with expectation. Let $D \subset A$ be a subalgebra with expectation. If $A \prec_M B$ then $D \prec_M B$. 
\end{prop}

The following generalization of the powerful intertwining theorem of S. Popa will be needed for Lemma \ref{alternative} which is crucial for the proof of Theorem \ref{solid}.

\begin{theo}[{\cite[Theorem 4.3]{HoudIsoUPF}}] \label{criterion_intertwine}
Let $M$ be any von Neumann algebra and $A \subset 1_AM1_A$, $B \subset 1_BM1_B$ two subalgebras with expectation. Suppose that $A$ is finite and choose a faithful normal conditional expectation $E_B: 1_BM1_B \rightarrow B$. Then the following are equivalent:
\begin{itemize}
	\item $A \nprec_M B$.
	\item There exists a net of unitaries $(u_i)_{i \in I}$ in $\mathcal{U}(A)$ such that
				\[ \forall x,y \in 1_BM1_A, \; E_B(xu_iy^*) \rightarrow 0 \]
				in the $*$-strong topology.
\end{itemize}
\end{theo}
\begin{rem}\label{density}
A useful fact is that it is sufficient to check the condition $E_B(xu_iy^*) \rightarrow 0$ only for $x,y$ in a dense subset of $1_BM1_A$ (see \cite[Theorem 4.3 (5)]{HoudIsoUPF}). Another useful trick is that if we have a family of subalgebras with expectation $(B_j)_{j \in J}$ such that $A \nprec_M B_j$ for all $j \in J$, then we can construct a net of unitaries $(u_i)_{i \in I}$ satisfying the condition simultaneously for all the subalgebras $B_j, j\in J$. To see this, one can first reduce to the case where $J$ is finite and then apply Theorem \ref{criterion_intertwine} to $A$ and $B=\bigoplus_{j \in J} B_j$, viewed as subalgebras of $\bigoplus_{j \in J} M$. This idea goes back to \cite[beginning of the proof of Theorem 3.4]{Ioana2008}.
\end{rem}

Finally, we consider an absorption property which was very prominent in Popa's work on deformation/rigidity, mainly because it helps to solve the relative commutant issue explained in Remark \ref{semi-conj}.
\begin{defs}
Let $M$ be a von Neumann algebra and $N \subset M$ a subalgebra with expectation. We say that $N \subset M$ is \emph{absorbing} if for every diffuse subalgebra with expectation $Q \subset 1_QN1_Q$, we have $Q' \cap 1_QM1_Q \subset 1_QN1_Q$.  
\end{defs} 

Obviously, the trivial inclusions $M \subset M$ and $\mathbb{C} \subset M$ are always absorbing (since $\mathbb{C}$ has no diffuse subalgebras). As it was observed by S. Popa, more interesting examples of absorbing inclusions are given by \emph{mixing} actions:

\begin{ex}[{\cite[Section 3]{PopaMalleable1}}]  \label{mixing}  Let $(A,\varphi)$ be a von Neumann algebra with a faithful normal state $\varphi$ and $\sigma: \Gamma \rightarrow \mathrm{Aut}(A,\varphi)$ a $\varphi$-preserving action of a discrete group $\Gamma$, let $M=A \rtimes_\sigma \Gamma$ be the crossed product von Neumann algebra. Suppose that the action $\sigma$ is \emph{mixing}, i.e.
\[ \forall a,b \in A, \; \lim_{g \rightarrow \infty} \varphi(a\sigma_g(b))=\varphi(a)\varphi(b). \] 

 Then the inclusion $\mathcal{L}(\Gamma) \subset M$ is absorbing. 
\end{ex}
 
 Other examples of this absorption phenomenon come from some group inclusions as well as free products \cite{popa1983orthogonal} and amalgamated free products \cite[Theorem 1.1]{Ioana2008}.
 
Later on, we will need the following lemma: 
\begin{lem} \label{abs_com}
Let $N \subset M$ be an absorbing inclusion. Let $Q \subset 1_QM1_Q$ be a diffuse subalgebra with expectation. If $Q \prec_M N$ then there exists a non-zero partial isometry $v \in M$ such that $v^*v \in Q \vee (Q' \cap 1_QM1_Q)$, $vv^* \in N$ and $v(Q \vee (Q' \cap 1_QM1_Q))v^* \subset N$. In particular, we have $Q \vee (Q' \cap 1_QM1_Q) \prec_M N$.
\end{lem}
\begin{proof}
Take non-zero projections $p \in Q$, $q \in N$, a subalgebra with expectation $C \subset qNq$  and a partial isometry $v \in qMp$ such that $pQp \sim_v C$. Then by Remark \ref{semi-conj}, we know that $v^*v \in (pQp)' \cap pMp$, $vv^* \in C' \cap qMq$ and 
\[ v(pQp \vee ((pQp)' \cap pMp))v^* \subset C \vee (C' \cap qMq) \]
 We have $pQp \vee ((pQp)' \cap pMp)=p(Q \vee (Q' \cap 1_QM1_Q))p$ so that $v^*v \in Q \vee (Q' \cap 1_QM1_Q)$ and $v(Q \vee (Q' \cap 1_QM1_Q))v^* \subset C \vee (C' \cap qMq)$.  Since $Q$ is diffuse, $C$ is also diffuse and  because $N$ is absorbing we have $C \vee (C' \cap qMq) \subset qNq$ so that we get $vv^* \in N$ and $v(Q \vee (Q' \cap 1_QM1_Q))v^* \subset N$. 
\end{proof}

\section{Relative solidity}
In \cite{TakaSolid}, N. Ozawa introduced the notion of \emph{solid} von Neumann algebras. It is easy to check that in the $\mathrm{II}_1$ case, the definition of solidity given in \cite{TakaSolid} is equivalent to the following one (see Proposition \ref{formulation}). 
\begin{defs}
Let $M$ be a von Neumann algebra. We say that $M$ is \emph{solid} if every properly non-amenable subalgebra with expectation $Q \subset 1_QM1_Q$ has discrete center.
\end{defs}

Equivalently, $M$ is \emph{solid} if and only if every subalgebra with expectation $Q \subset M$ is a direct sum of an amenable von Neumann algebra and a family (possibly empty) of non-amenable factors. This clearly shows the analogy with the notion of \emph{solid ergodicity} for equivalence relation discovered in \cite{ChifanIoanaBern} and formally introduced in \cite[Definition 5.4]{GaboriauSurvey}.

In this section, we are interested in a \emph{relative} version of solidity:

\begin{defs}
Let $M$ be a von Neumann algebra and $N \subset M$ a subalgebra with expectation. We say that $M$ is \emph{solid relatively to} $N$ if every properly non-amenable subalgebra with expectation $Q \subset 1_QM1_Q$ such that $Q \nprec_M N$ has discrete center.
\end{defs}

Clearly, a von Neumann algebra $M$ is solid if and only if it is solid relatively to $\mathbb{C}$. The following property justifies the terminology.

\begin{prop} \label{relative_absolute}
Let $P \subset N \subset M$ be inclusions of von Neumann algebras with expectations. If $M$ is solid relatively to $N$ and $N$ is solid relatively to $P$ then $M$ is solid relatively to $P$. In particular, if $M$ is solid relatively to $N$ and $N$ is solid then $M$ is solid.
\end{prop}
\begin{proof} Take a properly non-amenable subalgebra with expectation $Q \subset 1_QM1_Q$ such that $Q$ has diffuse center. We have to show that $Q \prec_M P$. Since $M$ is solid relatively to $N$, we already know that $Q \prec_M N$. Hence there exist a non-zero projection $p \in Q$, a subalgebra with expectation $C \subset 1_CN1_C$ and a partial isometry $v \in 1_CMp$ such that $pQp \sim_v C$. Let $\psi: pQp \rightarrow C$ be the $*$-isomorphism by $v$. Let $q$ be the smallest projection in $N$ which is greater then $vv^*$ so that we have $yv=0 \Leftrightarrow yq=0$ for all $y \in N$. Since $vv^* \in C' \cap 1_CM1_C$, we have $q \in C' \cap 1_CN1_C$. Hence, we have a  $*$-morphism $\phi: pQp \rightarrow Cq$ defined by $\phi(x)=\psi(x)q$. In fact, it is a $*$-isomorphism because $\psi(x)q=0 \Leftrightarrow \psi(x)v=0 \Leftrightarrow vx=0 \Leftrightarrow x=0$. Now $Cq \subset qNq$ is a sublagebra with expectation which is properly non-amenable and it has diffuse center (because it is isomorphic to $pQp$). Since $N$ is solid relatively to $P$, this implies that $Cq \prec_N P$. Hence there exists a subalgebra with expectation $D \subset 1_DP1_D$, a non-zero projection $r \in Cq$ and a partial isometry $w \in 1_DNr$ such that $rCr \sim_w D$. Let $\theta: rCr \rightarrow D$ be the $*$-isomorphism implemented by $w$. Let $p'=\phi^{-1}(r)$, let $\phi': p'Qp' \rightarrow rCr$ be the $*$-isomorphism obtained by restriction of $\phi$ to $p'Qp'$ and let $h:=wv$. Then we have a $*$-isomorphism $\alpha:=\theta \circ \phi': p'Qp' \rightarrow D$ and for all $x \in p'Qp'$ we have 
\[ \alpha(x)h=\theta(\phi'(x))wv=w\phi'(x)v=wvx=hx \]
Moreover, if $hx=0$ then $\alpha(x)wv=0$ which means that $\alpha(x)w=\alpha(x)wq=0$ by definition of $q$. This implies that $\theta(\phi'(x))=0$ by definition of $w$ and hence $x=0$. Thus we proved that $p'Qp' \sim_h D$ which means that $Q \prec_M P$ as we wanted.
\end{proof}

Next we present other possible formulations of relative solidity.

\begin{prop} \label{formulation}
Let $M$ be a von Neumann algebra and $N \subset M$ a subalgebra with expectation. The following are equivalent:
\begin{enumerate}
	\item $M$ is solid relatively to $N$.
	\item For every diffuse subalgebra with expectation $Q \subset 1_QM1_Q$ such that $Q' \cap 1_QM1_Q$ is non-amenable we have $Q' \cap 1_QM1_Q \prec_M N$.
	\item For every non-amenable subalgebra with expectation $Q \subset 1_QM1_Q$ such that $Q' \cap 1_QM1_Q$ is diffuse we have $Q \prec_M N$.
\end{enumerate}
\end{prop}
\begin{proof}
$(1) \Rightarrow (2)$. Suppose that $Q' \cap 1_QM1_Q \nprec_M N$. Take $A \subset Q$ a diffuse abelian subalgebra with expectation (see \cite[Lemma 2.1]{HoudUedaAsympFree}). Then $P = A \vee (Q' \cap 1_QM1_Q)$ has diffuse center and $P \nprec_M N$ by Proposition \ref{sub_embed}. Hence $P$ is amenable by $(1)$. Therefore $Q' \cap 1_QM1_Q$ is also amenable.

$(2) \Rightarrow (3)$. Let $P = Q' \cap 1_QM1_Q$. Then $P$ is diffuse and since $Q$ is non-amenable and $Q \subset P' \cap 1_QM1_Q$ we have that $P' \cap 1_QM1_Q$ is non-amenable. Hence $P' \cap 1_QM1_Q \prec_M N$ by $(2)$. Therefore $Q \prec_M N$ by Proposition \ref{sub_embed}.

$(3) \Rightarrow (1)$. Let $Q \subset 1_QM1_Q$ be a properly non-amenable subalgebra with expectation such that $Q$ has diffuse center. Then $Q' \cap 1_QM1_Q$ is also diffuse. Hence by $(3)$, we have $Q \nprec_M N$. Therefore $M$ is solid relatively to $N$.
\end{proof}

 Note that \cite[Theorem 2]{ChifanIoanaBern} as well as \cite[Theorem B]{RemiGaussian} and \cite[Theorem C]{RemiGaussian} are examples of relative solidity results. Relative solidity can be useful to prove fullness or primeness in situations where true solidity fails. Recall that a factor $M$ is \emph{full} (\cite{ConnesAlmostPeriodic}) if for every bounded net $x_i \in M, i \in I$ such that $||[x_i,\varphi] || \rightarrow 0$ for all $\varphi \in M_*$ there exists a net $\lambda_i \in \C, i \in I$ such that $x_i - \lambda_i \rightarrow 0$ in the $*$-strong topology. A factor $M$ is \emph{prime} if it is not of type $\mathrm{I}$ and $M \simeq P_1 \ovt P_2$ implies that $P_1$ or $P_2$ is of type $\mathrm{I}$.

\begin{prop}\label{prime_full}
Let $M$ be a von Neumann algebra and $N \subset M$ a subalgebra with expectation. Suppose that $M$ is solid relatively to $N$ and the inclusion $N \subset M$ is absorbing. Let $P \subset 1_PM1_P$ be any non-amenable factor with expectation such that $P \nprec_M N$. Then $P$ is prime. If $P$ has moreover a separable predual then it is full.
\end{prop}

\begin{proof}
Suppose that $P=P_1 \ovt P_2$ where $P_1$ and $P_2$ are two diffuse factors. Since $P$ is non-amenable then one of them, say $P_1$, is non-amenable. Since $M$ is solid relatively to $N$, by Proposition \ref{formulation}, we must have $P_2 \prec_M N$. Since $N \subset M$ is absorbing, this implies that $P=P_2 \vee (P_2'\cap P) \prec_M N$ by Lemma \ref{abs_com}. Contradiction. 

Now we suppose that $P$ has separable predual and we show that $P$ is full. On the contrary, suppose that $P$ is not full. Then, by \cite[Theorem 3.1]{HoudUedaFreeRig}, there exists a decreasing sequence of diffuse abelian subalgebras $Q_i \subset P$ with expectation such that $P = \bigvee_{i \in \mathbb{N}} (Q_i' \cap P)$. Suppose that for some $i$, we have $Q_i'\cap 1_PM1_P \prec_M N$. Note that $Q_i \subset Q_i'\cap 1_PM1_P$. So by Lemma \ref{abs_com}, we know that there exists a non-zero partial isometry $v \in M$ such that $e=v^*v \in Q_i' \cap 1_PM1_P$, $f=vv^* \in N$ and $v(Q_i' \cap qMq)v^* \subset fNf$ with expectation. Note that for all $j \geq i$, we have $Q_j \subset Q_i \subset Q_i' \cap 1_PM1_P$ with expectation and $Q_j$ is diffuse. Hence, since $N$ is absorbing we have $v(Q'_j \cap 1_PM1_P)v^* \subset (vQ_jv^*)' \cap fMf \subset fNf$. Therefore for all $j \geq i$, we have $v(Q'_j \cap 1_PM1_P)v^* \subset fNf$. Thus $v(\bigvee_{i \in \mathbb{N}} (Q_i' \cap 1_PM1_P))v^* \subset fNf$. In particular $\bigvee_{i \in \mathbb{N}} (Q_i' \cap 1_PM1_P) \prec_M N$. Since $P \subset \bigvee_{i \in \mathbb{N}} (Q_i' \cap 1_PM1_P)$ with expectation we get $P \prec_M N$ by Proposition \ref{sub_embed}. But this is not possible by assumption on $P$. Hence we must have $Q_i' \cap 1_PM1_P \nprec_M N$ for all $i$. By relative solidity and using Proposition \ref{formulation}, this implies that $Q_i' \cap 1_PM1_P$ is amenable for all $i$. In particular, $Q_i' \cap P$ is amenable for all $i$. Hence $P=\bigvee_{i \in \mathbb{N}} (Q_i' \cap P)$ is also amenable. From this contradiction we conclude that $P$ is full.
\end{proof}

\section{Spectral gap rigidity for non-tracial von Neumann algebras}
In this section, we prove an abstract non-tracial version of Popa's spectral gap rigidity principle \cite[Lemma 5.1]{PopaSpectralGap} and \cite[Lemma 5.2]{PopaSpectralGap}. The idea is that in the presence of a \emph{spectral gap} property, a subalgebra with properly non-amenable commutant will behave as a \emph{rigid} subalgebra, making it easy to locate. This principle will be used to prove Theorem \ref{solid} in the next section. 

\subsection*{Bimodules and the spectral gap property}
Let $M$ be a von Neumann algebra. As usual, an \emph{$M$-$M$-bimodule} is a pair $(H,\pi_H)$ where $H$ is a Hilbert space and $\pi_H: M \otimes_{alg} M^{op} \rightarrow B(H)$ is a $*$-representation which is \emph{binormal} (i.e.\ the restrictions of $\pi_H$ to $M$ and $M^{op}$ are normal). Recall that if $H,K$ are $M$-$M$-bimodules with bimodule representations $\pi_H: M \otimes_{alg} M^{op} \rightarrow B(H)$ and $\pi_K: M \otimes_{alg} M^{op} \rightarrow B(K)$ then $K$ is \emph{weakly contained} in $H$ if and only if there exists a unital completely positive map $\Phi: B(H) \rightarrow B(K)$ such that $\Phi \circ \pi_H = \pi_K$ (Note that $\Phi$ will restrict to a morphism from the $C^*$-algebra generated by $\pi_H(M \otimes_{alg} M^{op})$ to the $C^*$-algebra generated by $\pi_K(M \otimes_{alg} M^{op})$, and conversely, such a morphism can be extended to a u.c.p map from $B(H)$ to $B(K)$ by Arveson's theorem). Recall also that a von Neumann algebra $M$ is amenable if and only if the \emph{identity} $M$-$M$-bimodule $L^2(M)$ is weakly contained in the \emph{coarse} $M$-$M$-bimodule $L^2(M) \ovt L^2(M)$. See \cite[Appendix B]{NCG} for more information on bimodules.

Let $M \subset N$ be an inclusion of von Neumann algebras with expectation. Then any choice of a faithful normal conditional expectation $E_M: N \rightarrow M$ gives rise to an inclusions of $M$-$M$-bimodules $L^2(M) \subset L^2(N)$. We say that the inclusion $M \subset N$ is \emph{coarse}\footnote{We thank R. Boutonnet for suggesting this name.} if, for some choice of a faithful normal conditional expectation $E_M: N \rightarrow M$, the $M$-$M$-bimodule 
\[ L^2(N) \ominus L^2(M)=\{ \xi \in L^2(N) \mid \xi \perp L^2(M) \} \] is weakly contained in the coarse $M$-$M$-bimodule $L^2(M) \ovt L^2(M)$.  

The following lemma is an abstract non-tracial version of an argument used in \cite[Lemma 5.1]{PopaSpectralGap}. See also \cite[Theorem 4.1]{HoudIsoStrongErg}.

\begin{lem}[Spectral gap] \label{ultra_gap}
Let $M \subset N$ be a coarse inclusion. Let $P \subset M$ be a subalgebra with expectation and suppose that $P$ is properly non-amenable. Let $\omega$ be any free ultrafilter on $\mathbb{N}$. Then we have $P' \cap N^\omega \subset M^\omega$.
\end{lem}
\begin{proof}
Let $E_M: N \rightarrow M$ be a faithful normal conditional expectation as in the definition of a coarse inclusion. Let $E^{\omega}: N^\omega \rightarrow N$ the canonical conditional expectation and $E_{M^\omega}: N^\omega \rightarrow M^\omega$ the conditional expectation induced by $E_M$.
 
Now, suppose, by contradiction, that there is $Y \in  P' \cap N^\omega$ with $Y \neq 0$ and such that $E_{M^\omega}(Y)=0$. We have $E_M(Y^*Y) \in P' \cap M$. Let $c \in P'\cap M$ be an element such that $q=E_M(Y^*Y)^{\frac{1}{2}}c$ is a non-zero projection in $P' \cap M$.
Then $Yc \in P' \cap N^\omega$ and $E_{M^\omega}(Yc)=0$ and we have $E_M((Yc)^*(Yc))=q$. So, without loss of generality, we can directly suppose that $q=E_M(Y^*Y) \in P' \cap M$ is a non-zero projection. We will show that the $P$-$P$-bimodule $qL^2(M)$ is weakly contained in the $P$-$P$-bimodule $L^2(N) \ominus L^2(M)$. Let $V : L^{2}(M) \rightarrow L^{2}(N)$ and  $W : L^2(N) \ominus L^2(M) \rightarrow L^{2}(N)$ be the inclusion of bimodules. Note that $VV^{*}=1-WW^{*}=e_M$.  Pick a sequence $(y_n)_{n \in \mathbb{N}}$ representing $Y$ and define a completely positive map 
\[ \Phi: B(L^2(N) \ominus L^2(M)) \rightarrow B(L^2(M))\]
\[ T \mapsto \lim_{n \rightarrow \omega} \left( V^{*} y_n^* W T W^{*}y_n V \right) \text{ in the weak* topology.}\]
 We have 
\[ \Phi(1)=\lim_{n \rightarrow \omega} \left( V^{*} y_n^* (1-e_M) y_n V \right)=E_M(Y^{*}Y)-E_M(Y^{*})E_M(Y)=q \]
Hence $\Phi$ takes its values in $qB(L^{2}(M)q \simeq B(qL^{2}(M))$ which means that we can view $\Phi$ as a unital completely positive map from $B(L^2(N) \ominus L^2(M))$ to $B(qL^2(M))$.
And since $Y \in P'\cap N^\omega$ we see that $\Phi$ preserves the $P$-$P$-bimodule representations. Hence the $P$-$P$- bimodule $qL^2(M)$ is weakly contained in $L^2(N) \ominus L^2(M)$. Since $P$ is with expectation in $M$, we have an inclusion of $P$-$P$-bimodules $L^2(P) \subset L^2(M)$. Hence the $P$-$P$-bimodule $qL^2(P)$ is weakly contained in $L^2(N) \ominus L^2(M)$. By the coarse inclusion property, this implies in particular that $qL^2(P)$ is weakly contained in the coarse $P$-$P$-bimodule. We conclude easily that $qP$ is amenable and this contradicts the assumption that $P$ is properly non-amenable.    
\end{proof}

\subsection*{Malleable deformations and spectral gap rigidity}
Symmetric malleable deformations, or \emph{s-malleable} deformations, where introduced by S. Popa in \cite{PopaMalleable1} as a very useful tool for obtaining intertwining relations. We start by recalling this notion. Let $M$ be a von Neumann algebra. A \emph{malleable deformation} of $M$ is a pair $(\widetilde{M},\theta)$ where $M \subset \widetilde{M}$ is an inclusion with expectation and $\theta: \mathbb{R} \rightarrow \mathrm{Aut}(\widetilde{M})$ is a continuous action of $\mathbb{R}$. The deformation $(\widetilde{M},\theta)$ is said to be \emph{symmetric} if there exists $\beta \in \mathrm{Aut}(\widetilde{M})$ such that $\beta_{|M}=\mathrm{Id}$ and for all $t \in \mathbb{R}$, $\beta \circ \theta_t \circ \beta= \theta_{-t}$. We will say that a subalgebra $Q \subset M$ with expectation is \emph{rigid relatively to} the deformation $(\widetilde{M},\theta)$ if $\theta$ converges uniformly on the unit ball of $Q$: for every $*$-strong neighborhood $\mathcal{V}$ of $0$ in $\widetilde{M}$ there exists $t_0 > 0$ such that
\[ \forall t \in [-t_0,t_0], \forall x \in (Q)_1, \; \theta_t(x)-x \in \mathcal{V}. \]

Now we can state the main theorem of this section. It is an abstract non-tracial version of an argument due to S. Popa \cite[Lemma 5.2]{PopaSpectralGap}.  
\begin{theo} \label{def_rigidity}
Let $M$ be a von Neumann algebra, $(\widetilde{M},\theta)$ a symmetric malleable deformation of $M$ and $Q \subset M$ a subalgebra with expectation. Suppose that
\begin{itemize}
	\item The inclusion $M \subset \widetilde{M}$ is coarse.
	\item $Q$ is finite.
	\item $Q' \cap 1_QM1_Q$ is properly non-amenable.
\end{itemize}

Then $Q$ is rigid relatively to the deformation $(\widetilde{M},\theta)$ and $Q \prec_{\widetilde{M}} \theta_1(Q)$.
\end{theo}

\begin{proof}

We proceed in 4 steps following the lines of Popa's original argument. In fact, only step 1 is different from the tracial case.

\emph{\textsc{Step 1} - The subalgebra $Q$ is rigid relatively to $(\widetilde{M},\theta)$.}

Suppose that $Q$ is not relatively rigid. Then we can find a $*$-strong neighborhood of $0$ in $\widetilde{M}$ denoted by $\mathcal{V}$, a sequence $x_n \in (Q)_1$ and a sequence of reals $t_n \rightarrow 0$ such that $\theta_{t_n}(x_n)-x_n \notin \mathcal{V}$ for all $n$. Let $\omega$ be any free ultrafilter on $\mathbb{N}$. Since $Q$ is finite the sequence $(x_n)_{n \in \mathbb{N}}$ defines an element $x$ in the ultraproduct $Q^\omega \subset M^\omega \subset \widetilde{M}^\omega$. Also, since $t_n \rightarrow 0$ we have $||\varphi \circ \theta_{t_n}  - \varphi || \rightarrow 0$ for all $\varphi \in \widetilde{M}_*$. Using this, we check that the automorphism of $\widetilde{M}$ defined by
\[\Theta((y_n)_n)=(\theta_{t_n/2}(y_n))_n \] preserves the ideal $\mathcal{I}^\omega(\widetilde{M})$. Hence $\Theta$ induces an automorphism of $\widetilde{M}^\omega$, still denoted $\Theta$, such that 
\[ \Theta((y_n)^\omega)=(\theta_{\frac{t_n}{2}}(y_n))^\omega. \]
Note that the choice of $x_n$ and $t_n$ we made implies that $\Theta^2(x) \neq x$. Now, observe that $\Theta(y)=y$ for all $y \in \widetilde{M}$ because $t_n \rightarrow 0$. In particular, if we let $P:=Q' \cap 1_QM1_Q$ then we have $\Theta(P)=P$ and therefore $\Theta(P' \cap 1_Q \widetilde{M}^\omega 1_Q)=P' \cap 1_Q \widetilde{M}^\omega 1_Q$. Since $x \in P' \cap 1_Q \widetilde{M}^\omega 1_Q$, we get $\Theta(x) \in P' \cap 1_Q \widetilde{M}^\omega 1_Q$. Lemma \ref{ultra_gap} applies to the coarse inclusion $1_QM1_Q \subset 1_Q \widetilde{M}1_Q$ and shows that
\[ P' \cap 1_Q \widetilde{M}^\omega 1_Q \subset 1_Q M^\omega 1_Q. \]
Therefore we get $\Theta(x) \in M^\omega$. Now, choose a symmetry $\beta$ for $(\widetilde{M},\theta)$ and extend it naturally to an automorphism $\beta \in \mathrm{Aut}(\widetilde{M}^\omega)$. Then $\beta$ fixes $M^\omega$ and we have $(\beta \circ \Theta \circ \beta)(x) = \Theta^{-1}(x)$. Since $x \in M^\omega$ and $\Theta(x) \in M^\omega$, we conclude that $\Theta(x)=\Theta^{-1}(x)$. And this contradicts the fact that $\Theta^2(x) \neq x$. Therefore $Q$ is rigid relatively to the deformation $(\widetilde{M},\theta)$.

\emph{\textsc{Step 2} - For sufficiently small $t$ there exists a non-zero $\theta_t(Q)$-$Q$-intertwiner.}

Since $Q$ is finite and with expectation, we can take $\psi$ a faithful normal state on $\widetilde{M}$ such that $Q$ is in the centralizer $\widetilde{M}^\psi$. By Step 1, $Q$ is rigid relatively to $(\widetilde{M},\theta)$. Hence we can find $t_0$ small enough so that for all $|t| \leq t_0$ we have
\[ \forall u \in \mathcal{U}(Q), \; \mathfrak{Re}( \psi(\theta_t(u)u^*)) \geq \frac{\psi(1_Q)}{2} > 0. \]
Now take $\mathcal{C} \subset (\widetilde{M})_1$ the weak$^*$ closed convex hull of $\{ \theta_t(u)u^* \mid u \in \mathcal{U}(Q) \}$ and let $w_t \in \mathcal{C}$ the unique element which minimizes $||w_t||_\psi$. Then, we have $w_t \in  \theta_t(1_Q)\widetilde{M}1_Q$ and by the above inequality $w_t \neq 0$. Since $Q \subset \widetilde{M}^\psi$ we have $||\theta_t(u)w_tu^*||_\psi = ||w_t||_\psi$ for all $u \in \mathcal{U}(Q)$. Therefore, by the uniqueness of $w_t$ we have
\[ \forall x \in Q, \; \theta_t(x)w_t = w_t x. \]
So $w_t$ is indeed a non-zero $\theta_t(Q)$-$Q$-intertwiner.

\emph{\textsc{Step 3} - If there exists a non-zero $\theta_t(Q)$-$Q$-intertwiner then there exists a non-zero $\theta_{2t}(Q)$-$Q$-intertwiner.}

Take a non-zero  $\theta_t(Q)$-$Q$-intertwiner $w_t \in \theta_t(1_Q)\widetilde{M}1_Q$ so that
\[ \forall x \in Q, \; \theta_t(x)w_t = w_t x.\]
Take a symmetry $\beta \in \mathrm{Aut}(\widetilde{M})$ for $(\widetilde{M},\theta)$. Let $P= Q' \cap 1_QM1_Q$. Note that for all $d \in P$, the element $w_t d \beta(w_t^*)$ is a $\theta_{t}(Q)$-$\theta_{-t}(Q)$-intertwiner. Indeed, for all $x \in Q$, we have
\[ \theta_{t}(x)w_t d \beta(w_t^*)=w_t d x\beta(w_t^*)=w_t d \beta(xw_t^*)=w_t d \beta(w_t^*\theta_t(x))=w_t d \beta(w_t^*)\theta_{-t}(x). \]
Hence, $\theta_t(w_t d \beta(w_t^*))$ is $\theta_{2t}(Q)$-$Q$-intertwiner. Therefore, we just need to find $d$ such that $w_t d \beta(w_t^*) \neq 0$. Suppose that $w_t d \beta(w_t^*) = 0$ for all $d \in P$. Let $q \in \widetilde{M}$ be the unique projection such that $\widetilde{M}q$ is the weak$^*$ closure of the left ideal $\widetilde{M}w_tP$. Then we have $q\beta(q)=0$. However, note that $q \in P' \cap 1_Q \widetilde{M}1_Q$ (because $\widetilde{M}q$ is invariant by the right multiplication by elements of $P$ so that $qx=qxq$ for all $x \in P$). Hence by Lemma \ref{ultra_gap}, we get that $q \in M$. Thus, we have $q=q\beta(q)=0$. Since $w_t \in \widetilde{M}q$, this contradicts the fact that $w_t \neq 0$ and we are done.

\emph{\textsc{Step 4} - Conclusion.}

Take $t=\frac{1}{2^n}$ sufficiently small and choose a non-zero $\theta_t(Q)$-$Q$-intertwiner. Then build recursively non-zero $\theta_{2^k t}(Q)$-$Q$-intertwiners until $k=n$. This gives a non-zero $\theta_1(Q)$-$Q$-intertwiner as we wanted.

\end{proof}

\section{Bernoulli crossed products}
In this section, we prove our main results using the spectral gap rigidity principle of the previous section. 

Fix $(A_0,\varphi_0)$ any von Neumann algebra with a faithful normal state $\varphi_0$. Let $(A,\varphi)=(A_0,\varphi_0)^{\ovt \Gamma}$ be the infinite tensor product indexed by $\Gamma$ where $\Gamma$ is any infinite discrete group.
Let $\sigma: \Gamma \rightarrow \mathrm{Aut}(A)$ be the \emph{Bernoulli action} of $\Gamma$ on $A$ obtained by shifting the tensors. The crossed product von Neumann algebra $M=A \rtimes_\sigma \Gamma$ is called the \emph{noncommutative Bernoulli crossed product} of $\Gamma$ with base $(A_0,\varphi_0)$. If $A_0 \neq \mathbb{C}$ and $\Gamma$ is infinite, it is well known that the Bernoulli action $\sigma$ is \emph{ergodic} (i.e.\ the fixed point algebra $A^\sigma$ is trivial) and \emph{properly outer} (i.e.\ if $\sigma_g(x)v=vx$ for all $x \in A$ and some non-zero $v \in A$ then $g=1$) and so, in this case, $M=A \rtimes_\sigma \Gamma$ is a factor (see \cite{Vaes2015296} for proofs of this facts). There is a canonical faithful normal conditional expectation $E_A: M \rightarrow A$ allowing us to extend $\varphi$ to a faithful normal state $\varphi$ on $M$ by the formula $\varphi \circ E_A = \varphi$. The action $\sigma$ preserves the state $\varphi$ so that $\mathcal{L}(\Gamma)$ is contained in the centralizer $M^\varphi$. An important fact for us is that the Bernoulli action is \emph{mixing}:
\[ \forall a,b \in A, \; \varphi(a\sigma_g(b)) \rightarrow \varphi(a)\varphi(b) \text{ when } g \rightarrow \infty. \]
Therefore, the inclusion $\mathcal{L}(\Gamma) \subset M$ is \emph{absorbing} (Example \ref{mixing}).

Now, following \cite{ChifanIoanaBern}, we define a symmetric malleable deformation of $M$. Let $(\widetilde{A}_0,\psi_0)=(A_0,\varphi_0) \ast ( \mathcal{L}(\mathbb{Z}),\tau)$ be the free product von Neumann algebra where $\tau$ is the Haar trace of $\mathcal{L}(\mathbb{Z})$. As before, let $(\widetilde{A},\psi) = (\widetilde{A}_0, \psi_0)^{\ovt \Gamma}$ be the infinite tensor product and $\widetilde{M}=\widetilde{A} \rtimes \Gamma$ the crossed product with respect to the Bernoulli action. The von Neumann algebra $\widetilde{M}$ contains $M$ with a normal conditional expectation $E_M: \widetilde{M} \rightarrow M$ such that $\psi =\varphi \circ E_M$. Now, take $v$ the canonical Haar unitary generating $\mathcal{L}(\mathbb{Z})$ and let $h \in \mathcal{L}(\Z)$ be a self-adjoint element such $v=e^{ih}$. For every $t \in \R$, let $v_t = e^{ith}$ and define an automorphism $\theta^0_t \in \mathrm{Aut}(\widetilde{A}_0)$ by $\theta^0_t(x)=v_txv_t^*$. Then $\theta^0_t$ induces an automorphism $\theta_t$ of $\widetilde{M}$ that fixes the elements of $\mathcal{L}(\Gamma)$ and preserves $\psi$. Moreover, the action  $\theta: t\mapsto \theta_t \in \mathrm{Aut}(\widetilde{M})$ is continuous. Let $\beta_0 \in \mathrm{Aut}(\widetilde{A}_0)$ be the automorphism defined by $\beta_0(a)=a$ for all $a \in A_0$ and $\beta_0(v)=v^*$. Then $\beta_0$ induces naturally an automorphism $\beta$ of $\widetilde{M}$ such that $\beta \circ \theta_t \circ \beta =\theta_{-t}$. Therefore $(\widetilde{M},\theta)$ is indeed a symmetric malleable deformation of $M$. In fact, it is malleable \emph{over} $\mathcal{L}(\Gamma)$ (see \cite[Section 6.2]{PopaSurvey}) meaning that we have the following \emph{commuting square relation} for the $\psi$-preserving conditional expectations:
\[ E_M \circ E_{\theta_1(M)} = E_{\theta_1(M)} \circ E_M = E_{\mathcal{L}(\Gamma)}. \]
This fact is important since in many cases it allows one to obtain an intertwining relation in $M$ from an intertwining relation in $\widetilde{M}$. In our specific case we get the following dichotomy which was obtained in the tracial case by S. Popa \cite[Theorem 4.1 and 4.4]{PopaMalleable1} and completed by A. Ioana \cite[Theorem 3.3 and 3.6]{IoanaWreath} (see also \cite[Theorem 7.3, step 3]{bourbakiHoud}). In order to generalize it to the non-tracial case we use Theorem \ref{criterion_intertwine}.

\begin{lem} \label{alternative}
Let $Q \subset 1_QM1_Q$ be a finite subalgebra with expectation. If $Q \prec_{\widetilde{M}} \theta_1(M)$ then one of the following holds:
\begin{itemize}
	\item $Q \prec_M \mathcal{L}(\Gamma)$.
	\item $Q \prec_M \ovt_F A_0$ for some finite subset $F \subset \Gamma$.
\end{itemize}
\end{lem}
\begin{proof}
Suppose that $Q \nprec_M \mathcal{L}(\Gamma)$ and $Q \nprec_M \ovt_F A_0$ for every finite subset $F \subset \Gamma$. Then by Theorem \ref{criterion_intertwine} and Remark \ref{density} we can find a net $u_i \in \mathcal{U}(Q)$ such that
\[ \forall x,y \in M, \; E_{\mathcal{L}(\Gamma)}(xu_iy^*)  \rightarrow 0 \]
\[ \forall F \subset \Gamma \text{ finite }, \forall x,y \in M, \; E_{\ovt_F A_0}(xu_iy^*) \rightarrow 0 \]
in the $*$-strong topology. We will contradict the assumption that $Q \prec_{\widetilde{M}} \theta_1(M)$ by showing that
\[ \forall x,y \in \widetilde{M}, \; E_{\theta_1(M)}(xu_iy^*) \rightarrow 0. \]
First, we can suppose that $x,y \in \widetilde{A}$ since $\mathcal{L}(\Gamma) \subset \theta_1(M)$ and $\widetilde{M}=\widetilde{A} \rtimes \Gamma$. By Remark \ref{density}, we can suppose that $x=\otimes_{g \in K} x_g$ and $y=\otimes_{g \in K'} y_g$ where $K,K' \subset \Gamma$ are finite subsets and $x_g,y_g \in \widetilde{A_0}$. Also the result is obvious if both $x,y \in \theta_1(A)A$ since $E_{\theta_1(M)} \circ E_M = E_{\mathcal{L}(\Gamma)}$. So we can suppose that $x$ is orthogonal to $\theta_1(A)A$ for example. Now, note that we have the following relation for all $g \in \Gamma$
\[ E_{\theta_1(A)}(E_{\theta_1(M)}(xu_iy^*)u_g^*)=E_{\theta_1(A)}(xE_{\ovt_{K \cup gK'} A_0}(u_i u_g^*)\sigma_g(y^*)). \]
This shows first that if $g$ is outside some finite set $F \subset \Gamma$, so that the support of $x$ and $\sigma_g(y)$ are disjoint, then 
\[ E_{\theta_1(A)}(E_{\theta_1(M)}(xu_iy^*)u_g^*)=0 \]
because $x$ is orthogonal to $\theta_1(A)A$. Hence we have a finite sum
\[ E_{\theta_1(M)}(xu_iy^*) = \sum_{g \in F} E_{\theta_1(A)}(E_{\theta_1(M)}(xu_iy^*)u_g^*) u_g \]
and each term of this sum converges to $0$ because $E_{\ovt_{K \cup gK'} A_0}(u_i u_g^*) \rightarrow 0$. Therefore we have $ E_{\theta_1(M)}(xu_iy^*) \rightarrow 0 $ for all $x,y \in \widetilde{M}$ as we wanted.
\end{proof}

In order to apply our deformation/rigidity principle, we still need to show, as in the original proof of Chifan and Ioana, that the inclusion $M \subset \tilde{M}$ is coarse, i.e.\ that the $M$-$M$-bimodule $L^{2}(\tilde{M}) \ominus L^{2}(M)$ is weakly contained in the coarse $M$-$M$-bimodule (see Section 4). 

\begin{theo} \label{bernoulli_gap}
Suppose that $A_0$ is amenable. Then the inclusion $M \subset \widetilde{M}$ is coarse.
\end{theo}
\begin{proof}
We will compute the bimodule $L^2(\widetilde{M})\ominus L^2(M)$ following \cite[Lemma 5]{ChifanIoanaBern}. The computations still hold even though $\psi$ is not a trace. We will use the bimodule notation $x \xi y \in L^2(M)$ for $x,y \in M$ and $\xi \in L^2(M)$. We will denote by $\psi^{\frac{1}{2}}$ the unique cyclic vector in $L^2(M)$ such that $\langle x \psi^{\frac{1}{2}}, \psi^{\frac{1}{2}} \rangle=\psi(x)$ for all $x \in M$.

Let $\mathcal{A}_0 \subseteq A_0$ be a $\varphi_0$-orthonormal base of $A_0$ with $1 \in \mathcal{A}_0$. In $\widetilde{A}_0=A_0 \ast \mathcal{L}(\mathbb{Z})$ consider the set
\[ \widetilde{\mathcal{A}}_0 = \{ v^{n_0}a_1 \cdots v^{n_{k-1}}a_k v^{n_k} \mid k \geq 0, \quad n_i \in \mathbb{Z} \setminus 0, \quad a_i \in \mathcal{A}_0 \setminus 1 \}. \]  

Then it is easy to check that the subspaces $A_0 \tilde{a}_0 A_0$ are pairwise $\psi_0$-orthogonal for $\tilde{a}_0 \in \widetilde{\mathcal{A}}_0$. Thus, we have
\[ L^2(\widetilde{A}_0) \ominus L^2(A_0) = \bigoplus_{\tilde{a}_0 \in \widetilde{\mathcal{A}}_0} \overline{A_0 \tilde{a}_0 A_0 \psi_0^{\frac{1}{2}}}. \]

Now, let $\widetilde{\mathcal{A}}$ be the set of elements of $\widetilde{A}$ of the form
\[ \tilde{a}=\otimes_{g \in \Gamma} \tilde{a}_g \]
where $\tilde{a}_g \in \widetilde{\mathcal{A}}_0$ for finitely many (and at least one) $g$ and $\tilde{a}_g=1$ otherwise. Then we have  
\[  L^2(\widetilde{A}) \ominus L^2(A) = \bigoplus_{ \tilde{a} \in \widetilde{\mathcal{A}}} \overline{A \tilde{a} A\psi^{\frac{1}{2}}}. \]

Now, focus on the $M$-$M$-bimodules $H_{\tilde{a}}=\overline{M\tilde{a}M\psi^{\frac{1}{2}}} \subset L^2(\widetilde{M})$ for $\tilde{a} \in \widetilde{\mathcal{A}}$. We note that $H_{\tilde{a}}=H_{\sigma_g(\tilde{a})}$ for all $g \in \Gamma$ while $H_{\tilde{a}}$ and $H_{\tilde{a}'}$ are orthogonal when $\tilde{a}$ and $\tilde{a'}$ are not in the same $\Gamma$-orbit. So let $\Omega$ be the set of $\Gamma$-orbits of $\widetilde{\mathcal{A}}$ and for every $\pi \in \Omega$ define 
\[ H_\pi:= H_{\tilde{a}} \]
where $\tilde{a}$ is any element of the orbit $\pi$. Then we have an $M$-$M$-bimodules decomposition
\[ L^2(\widetilde{M}) \ominus L^2(M) = \bigoplus_{\pi \in \Omega} H_\pi .\] 
In order to conclude, it suffices to show that for each $\pi \in \Omega$, $H_\pi$ is weakly contained in $L^2(M) \ovt L^2(M)$. So let $\pi$ be such an orbit, represented by $\tilde{a} \in \widetilde{\mathcal{A}}$. Write $\tilde{a}=\otimes_{g \in \Gamma}\tilde{a}_g$ and let $F$ the non-empty finite set of elements $g \in \Gamma$ such that $\tilde{a}_g \neq 1$. The stabilizer of $\tilde{a}$ inside $\Gamma$ is denoted by
\[ S = \{ g \in \Gamma \mid \sigma_g(\tilde{a})=\tilde{a} \}. \]
It is a finite subgroup since it must leave the support $F$ invariant. Now let 
\[ K = A_0^{\ovt \Gamma \setminus F} \rtimes S \subset M. \]
The von Neumann algebra $K$ is globally invariant by the modular flow of $\varphi$. So there exists a unique normal conditional expectation $E_{K}$ from $M$ to $K$ that preserves $\varphi$. We will show that there is an isomorphism of $M$-$M$-bimodules
\[ H_\pi=H_{\tilde{a}} \simeq L^2(\langle M, K \rangle) \]
where $\langle M, K \rangle=(J_M K J_M)' \subset B(L^2(M))$ is the basic construction. Let $e_{K} \in \langle M, K \rangle$ be the Jones projection associated to $E_K$. It is known that the $*$-subalgebra $Me_{K}M$ is dense in $\langle M, K \rangle$. Let $\hat{\varphi}$ be the unique normal faithful semi-finite weight on $\langle M, K \rangle$ which satisfies
\[ \forall x,y \in M, \hat{\varphi}(xe_{K}y)=\varphi(xy). \]
We have
\[ L^2(\langle M, K \rangle) = \overline{Me_{K}M \hat{\varphi}^{\frac{1}{2}}}. \]
Denote by $U: H_{\tilde{a}} \rightarrow L^2(\langle M, K \rangle)$ the linear map densely defined by
\[ U(x\tilde{a}y\psi^{\frac{1}{2}})=xe_{K}y\hat{\varphi}^{\frac{1}{2}} \]
for all $x,y \in M$ which are $\varphi$-analytic. We claim that $U$ extends to a unitary map which is a $M$-$M$-bimodule isomorphism. In fact $U$ clearly commutes with the left action, and it also commutes with the right action because for $z \in M$ analytic we have
\[ U(x\tilde{a}y\psi^{\frac{1}{2}}z)=U(x\tilde{a}yz'\psi^{\frac{1}{2}})=xe_{K}yz'\hat{\varphi}^{\frac{1}{2}}=xe_{K}y\hat{\varphi}^{\frac{1}{2}}z= U(x\tilde{a}y\psi^{\frac{1}{2}})z. \]
where $z' = \sigma^\varphi_{\frac{-i}{2}}( z)=\sigma^\psi_{\frac{-i}{2}}( z)=\sigma^{\hat{\varphi}}_{\frac{-i}{2}}( z)$ (Note that the modular flows of $\psi$ and $\hat{\varphi}$ coincide on $M$ with the modular flow of $\varphi$). So it only remains to check that $U$ defines indeed a unitary, i.e.\ that
\[  \psi(y_1^*\tilde{a}^*x_1^*x_2\tilde{a}y_2)=\hat{\varphi}(y_1^*e_{K}x_1^*x_2e_{K}y_2) \]
for all $x_1,x_2,y_1,y_2 \in M$ which are $\varphi$-analytic. Once again, since the modular flow of $\hat{\varphi}$ and $\psi$ coincide on $M$, we can pass $y_1^*$ to the other side 
\[  \psi(y_1^*\tilde{a}^*x_1^*x_2\tilde{a}y_2)=\psi(\tilde{a}^*x_1^*x_2\tilde{a}y_2\sigma^\varphi_{-i}(y_1^*)) \]
\[  \hat{\varphi}(y_1^*e_{K}x_1^*x_2e_{K}y_2)=\hat{\varphi}(e_{K}x_1^*x_2e_{K}y_2\sigma^\varphi_{-i}(y_1^*)) \]
and so we just need to check that
\[ \forall x,y \in M, \psi(\tilde{a}^*x\tilde{a}y)=\hat{\varphi}(e_{K}xe_{K}y)=\varphi(E_{K}(x)y). \]
In order to prove this, we can suppose, by density, that $x$ and $y$ are of the form
\[ x=(\otimes_{g}x_g)u_{\gamma} \] 
\[ y=(\otimes_{g}y_g)u_{\delta} \]
with $\gamma, \delta \in \Gamma$, $x_g,y_g \in A_0$ for all $g$ and $x_g=y_g=1$ except for finitely many $g$. We also have $\tilde{a}=\otimes_g \tilde{a}_g$ with $\tilde{a}_g \in \widetilde{\mathcal{A}}_0$ for finitely many (not zero) $g$ and $\tilde{a}_g=1$ otherwise. Recall that $\widetilde{\mathcal{A}}_0$ is our orthonormal base.

Now we compute $\psi(\tilde{a}^*x\tilde{a}y)$. First, if $\delta \gamma \neq 1$ then $\psi(\tilde{a}^*x\tilde{a}y)=0$. So suppose that $\delta=\gamma^{-1}$. Then we have
\[ \psi(\tilde{a}^*x\tilde{a}y)=\psi(\otimes_g (\tilde{a}_g^*x_g\tilde{a}_{\gamma^{-1}g}y_{\gamma^{-1}g}))=\prod_{g \in \Gamma} \psi_0(\tilde{a}_g^*x_g\tilde{a}_{\gamma^{-1}g}y_{\gamma^{-1}g}). \]
For this product to be non-zero, we must have $\tilde{a}_g=\tilde{a}_{\gamma^{-1}g}$ for all $g$. This means that $\sigma_\gamma(\tilde{a})=\tilde{a}$, i.e.\ $\gamma \in S$. 
In this case, the usual computation of free probability gives
\[ \psi(\tilde{a}^*x\tilde{a}y)=\prod_{g \in F} \varphi_0(x_g)\varphi_0(y_{\gamma^{-1}g}) \prod_{g \in \Gamma \setminus F} \varphi_0(x_g y_{\gamma^{-1}g}). \]
So we have shown that $\psi(\tilde{a}^*x\tilde{a}y)$ is given by the formula above when $\delta=\gamma^{-1} \in S$ and is equal to $0$ otherwise.

Now, in order to compute $\varphi(E_{K}(x)y)$, we just need to check the formula
\[ E_{K}(x)=E_{K}((\otimes_g x_g)u_{\gamma})=1_{S}(\gamma)\prod_{g \in F} \varphi_0(x_g) ( \otimes_{g \in \Gamma \setminus F} x_g )  u_\gamma \] 
and we conclude that the equality
\[ \psi(\tilde{a}^*x\tilde{a}y)=\varphi(E_{K}(x)y) \]
is true in all cases.

Hence, we have shown that there is an isomorphism of $M$-$M$-bimodules 
\[ H_\pi=H_{\tilde{a}} \simeq L^2(\langle M, K \rangle). \]
Finally, since $K=A_0^{\ovt \Gamma \setminus F} \rtimes S$ is the crossed product of an amenable von Neumann algebra by a finite group $S$ then $K$ is also amenable. Therefore, its commutant $\langle M, K \rangle$ is also amenable. In particular, this implies that $H_\pi$ is weakly contained in $L^2(M) \ovt L^2(M)$ as an $M$-$M$-bimodule. Since this is true for all $\pi \in \Omega$, we conclude that the $M$-$M$-bimodule
\[ L^2(\widetilde{M}) \ominus L^2(M) = \bigoplus_{\pi \in \Omega} H_\pi \] 
is weakly contained in $L^2(M) \ovt L^2(M)$ as we wanted.
\end{proof}

\begin{proof}[Proof of Theorem \ref{solid}] We first show that $M$ is solid relatively to $\mathcal{L}(\Gamma)$. Suppose we have a properly non-amenable subalgebra with expectation $Q \subset 1_QM1_Q$ such that $Z=\mathcal{Z}(Q)$ is diffuse. We have to show that $Q \prec_M \mathcal{L}(\Gamma)$. Since $\mathcal{L}(\Gamma) \subset M$ is absorbing and $Q \subset Z' \cap 1_QM1_Q$, then using Lemma \ref{abs_com} and Proposition \ref{sub_embed}, we see that it is enough to show that $Z \prec_M \mathcal{L}(\Gamma)$. By Theorem \ref{def_rigidity} we know that $Z \prec_{\widetilde{M}} \theta_1(M)$. Hence by Lemma \ref{alternative} we must have $Z \prec_M \mathcal{L}(\Gamma)$ or $Z \prec_M \ovt_F A_0$ for some finite subset $F \subset \Gamma$. So we just need to show that the case where $Z \prec_M \ovt_F A_0$ leads to a contradiction. Take a subalgebra with expectation $C \subset 1_C \left( \ovt_F A_0 \right) 1_C$ and a non-zero projection $p \in Z$ such that $Zp \sim_M C$. Note that $Z$ and $C$ are abelian so, by Remark \ref{semi-conj}, we know that $Z' \cap 1_QM1_Q \prec_M C' \cap 1_CM1_C$. Hence, by Proposition \ref{sub_embed}, we get $Q \prec_M C' \cap 1_CM1_C$. Recall that, by assumption, $Q$ is properly non-amenable. Therefore, we just need to show that $C' \cap 1_CM1_C$ is amenable in order to get a contradiction. Let $x \in C' \cap 1_C M 1_C$. We claim that $x_g=E_{A}(xu_g^*)=0$ for $g \notin FF^{-1}$. In fact, since $C$ is abelian and diffuse (because $Z$ is abelian and diffuse), there exists a net of unitaries $u_i \in \mathcal{U}(C)$ which tends weakly to $0$. 
Take $g \notin FF^{-1}$. Then we have $\sigma_g(u_i) \in \ovt_{\Gamma \setminus F} A_0$. Hence
\[\forall a,b\in A, \; E_{\ovt_F A_0}(a\sigma_g(u_i)b) \rightarrow 0 \]
in the $*$-strong topology. Since $u_i x_g = x_g \sigma_g(u_i)$ we get
\[ u_iE_{\ovt_F A_0}(x_gx_g^*)= E_{\ovt_F A_0}(x_g\sigma_g(u_i)x_g^*) \rightarrow 0 \] 
in the $*$-strong topology. Thus $x_g=0$. Therefore we have shown that
\[ C' \cap 1_CM1_C \subset \sum_{g \in FF^{-1}} Au_g \]
We will show that this implies that $C' \cap 1_CM1_C$ is amenable. Let $D=C' \cap 1_CM1_C \oplus (1-1_C)\C$. Let $E_D: M \rightarrow D$ be a faithful normal conditional expectation. Since $A$ is amenable, there is a conditional expectation $\Psi: B(L^2(A)) \rightarrow A$. Define a map $\Phi: B(L^2(M)) \rightarrow D$ by
\[ \Phi(T) = \sum_{g \in FF^{-1}} E_D(\Psi(e_ATu_g^*e_A)u_g) \]
where $e_A: L^2(M) \rightarrow L^2(A)$ is the Jones projection. Since $D \subset \sum_{g \in FF^{-1}} Au_g$, a little computation shows that $\Phi(x)=x$ for all $x \in D$. Moreover $\Phi$ is completely bounded (compositions and finite sums of completely bounded maps are still completely bounded). Therefore, using \cite[Corollaire 5]{Pisier}, we have that $D$ is amenable which means that $C' \cap 1_C M 1_C$ is amenable as we wanted.

For the second part of the theorem we can apply Proposition \ref{relative_absolute} and Proposition \ref{prime_full} to get the desired conclusion. Note that when $\Gamma$ is non-amenable and $A_0 \neq \mathbb{C}$ then $M$ is a non-amenable factor and $M \nprec_M \mathcal{L}(\Gamma)$. Indeed, since $A_0 \neq \C$ then $A$ is diffuse and we can find a diffuse abelian subalgebra with expectation $B \subset A$. Let $u_i \in B$ a net of unitaries wich tends weakly to $0$. Then we have $E_{\mathcal{L}(\Gamma)}(x u_i y) \rightarrow 0$ in the $*$-strong topology for all $x,y \in M$ (by density, it suffices to check it for $x,y$ of the form $a u_g$ with $g \in \Gamma$ and $a \in A$). Therefore, by Theorem \ref{criterion_intertwine}, we know that $B \nprec_M \mathcal{L}(\Gamma)$. By Proposition \ref{sub_embed}, we get $M \nprec_M \mathcal{L}(\Gamma)$.

\end{proof}
\begin{rem} Note that the separability and countability assumptions in Theorem \ref{solid} are only needed for the fullness property in Proposition \ref{prime_full}. The proof of the relative solidity in itself does not require any separability assumption. 
\end{rem}

\begin{proof}[Proof of Corollary \ref{invariants}]
See (\cite{ConnesAlmostPeriodic}, Proposition 3.9) and the constructions in (\cite{ConnesAlmostPeriodic}, Corollary 4.4) for the first part and (\cite{ConnesAlmostPeriodic}, Theorem 5.2) for the second part. In both cases, the examples are obtained by taking Bernoulli crossed products $M=(A_0,\varphi_0) \rtimes \mathbb{F}_2$ where $\mathbb{F}_2$ is the free group on $2$ generators and $A_0$ is some non-trivial amenable algebra with separable predual. Since $\mathcal{L}(\mathbb{F}_2)$ is solid and non-amenable then by Theorem \ref{solid} we know that $M$ is solid and non-amenable. If $B_0$ is a Cartan subalgebra of $A_0$, then it is not hard to check (using the fact that the Bernoulli action is properly outer) that $\ovt_{\mathbb{F}_2} B_0$ is again a Cartan subalgebra of $M$. 
\end{proof}

Before the proof of the last corollary, we refer to \cite{FeldMooreII} for the construction of the von Neumann algebra $\mathcal{L}(\mathcal{R})$ of a non-singular equivalence relation $\mathcal{R}$ and its properties. 

\begin{proof}[Proof of Corollary \ref{equivalence}]
We can suppose that $X_0$ is not a point and that $\Gamma$ is infinite (otherwise the result is obvious). Let $X=(X_0,\mu_0)^\Gamma$. Let $A_0=\mathcal{L}(\mathcal{R}_0)$ and let $\varphi_0$ be the faithful normal state on $A_0$ induced by $\mu_0$. Let $B_0=L^\infty(X_0) \subset A_0$ be the canonical Cartan subalgebra. Since $A_0 \neq \mathbb{C}$ and $\Gamma$ is infinite, the Bernoulli action of $\Gamma$ on $(A_0,\varphi_0)^{\ovt \Gamma}$ is properly outer. Using this, it is not hard to check that $(B_0,\varphi_0)^{\ovt \Gamma} \subset M=(A_0,\varphi_0)^{\ovt \Gamma} \rtimes \Gamma$ is again a Cartan subalgebra and that we can identify canonically the Cartan pair $(M,(B_0,\varphi_0)^{\ovt \Gamma})$ with the Cartan pair $(\mathcal{L}(\mathcal{R}),L^\infty(X))$. 

Let $\mathcal{S} \subset \mathcal{R}$ be a subequivalence relation. Then with the preceding identification we have that $(B_0,\varphi_0)^{\ovt \Gamma} \subset \mathcal{L}(\mathcal{S}) \subset M$ with expectations. Since the subalgebra $(B_0,\varphi_0)^{\ovt \Gamma} \subset (A_0,\varphi_0)^{\ovt \Gamma}$ is diffuse, it is easy to check that $(B_0,\varphi_0)^{\ovt \Gamma} \nprec_M \mathcal{L}(\Gamma)$. Therefore $\mathcal{L}(\mathcal{S}) \nprec_M \mathcal{L}(\Gamma)$. Hence, Theorem \ref{solid} applies and we get a sequence of projections $z_n \in \mathcal{Z}(\mathcal{L}(\mathcal{S}))$ with $\sum_n z_n=1$ such that $\mathcal{L}(\mathcal{S})z_0$ is amenable and $\mathcal{L}(\mathcal{S})z_n$ is a full prime factor for all $n \geq 1$. By identifying the projections $z_n$ with $\mathcal{S}$-invariant measurable subsets $Z_n \subset X$ we get the desired conclusion since fullness and primeness of $\mathcal{L}(\mathcal{S}_{|Z_n})=\mathcal{L}(\mathcal{S})z_n$ imply easily strong ergodicity and primeness of $\mathcal{S}_{|Z_n}$.
\end{proof}

\bibliography{database}

\end{document}